\def\version{May 24, 2019}

\documentclass[12pt]{article}

\usepackage{amsfonts}
\usepackage{amsmath,amssymb,amsthm}
\usepackage{appendix}
\usepackage{bbm} % used in \newcommand{\1}{\mathbbm{1}}
\usepackage{amsbsy}
\usepackage{enumerate}
\usepackage{cite}
\usepackage{enumerate}
\usepackage[dvips]{graphicx}
\usepackage{url}
\usepackage{hyperref}
\hypersetup{colorlinks,citecolor=blue,linkcolor=blue,urlcolor=black}

\usepackage[textwidth=500pt,textheight=650pt,centering]{geometry} % 12pt

\def\UseSection{%%
        \numberwithin{equation}{section}
	\theoremstyle{plain}% default theorem style
        \newtheorem{theorem}    {Theorem}[section]
        \DefineTheorems % Use this to define other environments to be
}

\def\DefineTheorems{%%
	\newtheorem{lemma}      [theorem] {Lemma}
	\newtheorem{prop}       [theorem] {Proposition}
	\newtheorem{cor}        [theorem] {Corollary}

	\theoremstyle{definition}% ``defn'' theorem style

	\newtheorem{example}       [theorem] {Example}

	\theoremstyle{definition}% ``remark'' theorem style

}
\UseSection   %Necessary to define Numbering scheme for Theorem, etc.

\newcommand{\1}{\mathbbm{1}}

\newcommand{\nnb}	{\nonumber \\}

\newcommand{\lbeq}[1]  {\label{e:#1}}
\newcommand{\refeq}[1] {\eqref{e:#1}}    % AMS-LaTeX trick!

\newcommand{\Ebold} {{\mathbb E}}

\newcommand{\Nbold} {{\mathbb N}}

\newcommand{\Pbold} {{\mathbb P}}

\newcommand{\Rbold} {{\mathbb R}}

\newcommand{\Zbold} {{\mathbb Z}}

\newcommand{\Pihat} {\mbox{${\hat{\Pi}}$}}

\newcommand{\Rd}    {{ {\Rbold}^d}}
\newcommand{\Zd}    {{ {\Zbold}^d }}
\newcommand{\R}{\Rbold}
\newcommand{\Z}{\Zbold}
\newcommand{\N}{\Nbold}

\newcommand{\ddp}[2]{\frac{\partial #1}{\partial #2}}

\usepackage[usenames]{color}

\renewcommand{\to} {\rightarrow}

 \title {
   Spatial moments for high-dimensional critical
   \\
   contact process, oriented percolation
   and lattice trees
 }

 \author{
   Akira Sakai\thanks{Faculty of Science,
     Hokkaido University, Nishi 8-chome, Kita 10-jo, Kita-ku, Sapporo,
    Hokkaido 060-0810, Japan.
     E-mail: {\tt sakai@math.sci.hokudai.ac.jp},
     \url{https://orcid.org/0000-0003-0943-7842}}
   \, and
   Gordon Slade\thanks{Department of Mathematics,
     University of British Columbia,
     Vancouver, BC, Canada V6T 1Z2.
     E-mail: {\tt slade@math.ubc.ca},
     \url{https://orcid.org/0000-0001-9389-9497}
     }}

\date\version

\begin{document}
\maketitle

\begin{abstract}
Recently, Holmes and Perkins identified conditions which ensure
that for a class of critical lattice models the scaling limit of the range
is the range of super-Brownian motion.
One of their conditions
is an estimate on a spatial moment of order higher than four, which
they verified for the sixth moment for spread-out lattice trees in dimensions $d>8$.
Chen and Sakai have proved the required moment estimate for spread-out
critical oriented percolation in dimensions $d+1>4+1$.
We prove estimates on all moments for the spread-out critical
contact process in dimensions $d>4$, which in particular
fulfills the spatial moment condition of Holmes and Perkins.
Our method of proof is relatively
simple, and, as we show, it applies also to oriented percolation and lattice trees.
Via the convergence results of Holmes and Perkins, the upper bounds on
the spatial moments can in fact be promoted to asymptotic formulas
with explicit constants.
\end{abstract}

\section{Introduction and results}

\subsection{Introduction}

It is by now well established that super-Brownian motion arises as the scaling
limit in a number of critical lattice models above the upper critical dimension,
including the voter model, the contact process, oriented percolation, percolation,
and lattice trees (see, e.g., \cite{HH17book,Slad02,Slad06,DP99,CDP00}).  There are various ways,
of differing strengths, of stating such convergence results.
A particularly strong statement is that the scaling limit of the
range of the critical lattice model is the range of super-Brownian motion,
with the convergence with respect to the Hausdorff metric on the set of
compact subsets of $\R^d$.
Recently, Holmes and Perkins
\cite{HP19,HP20} have identified conditions which imply this strong form of convergence.
These conditions also imply an asymptotic formula for the probability of
exiting a large ball, i.e., for the extrinsic one-arm
probability.

One of the substantial conditions of \cite{HP19} is an estimate on a spatial
moment of degree higher than four.
Holmes and Perkins have proved the required bound on the sixth moment for
spread-out lattice
trees in dimensions $d>8$, and Chen and Sakai have proved asymptotic formulas
for \emph{all} spatial moments for spread-out critical oriented percolation
in dimensions $d+1>4+1$ (their case $\alpha=\infty$, see \cite[p.510]{CS11}).
These results are sufficient
to establish the spatial moment condition of Holmes and Perkins in these contexts.

In this paper, we prove, in a unified and relatively
simple fashion, estimates
on all spatial moments for critical spread-out models
of the contact process in dimensions $d>4$, oriented percolation in
dimensions $d+1>4+1$, and lattice trees in dimensions $d>8$.
Our results for the contact process are new, and,
together with the other conditions verified in \cite{HP19,HP20},
yield the conclusions of Holmes and Perkins.
As one of their conclusions,
Holmes and Perkins have shown that the upper bounds on the spatial moments can be
promoted to asymptotic formulas.
We discuss this last point in detail in Section~\ref{sec:af}.

Our method of proof
simplifies the bound of \cite{HP19}
on the sixth moment for lattice trees and extends it to all moments.
It also provides
a simpler approach to bounding the moments for oriented percolation than the
method of \cite{CS11} (who however did obtain
asymptotic formulas for the moments).
Our proof is based on the lace expansion, which has been applied to
study
critical models of the contact process
(e.g., \cite{Saka01,HS04z,HS10}),
oriented percolation (e.g., \cite{NY93,NY95,HS03b,CS08,CS09a,CS11})
and lattice trees (e.g., \cite{HS92c,DS98,Holm08,Holm16}).
We emphasise the contact process throughout the paper, because it is the
greatest novelty in our work, and because it is the most delicate of
the three models to analyse.

Our spread-out models are formulated in terms of a probability measure
$D$ on $\Z^d$, which is defined as follows.  Let
$h:\Rd\to[0,\infty)$ be bounded, continuous almost everywhere, invariant under
the symmetries of $\Zd$, and such that
\begin{align}
\int h(x)~\text{d}^dx=1,&&
\int|x|^nh(x)~\text{d}^dx<\infty\quad (n\in\N).
\end{align}
Given $L\ge 1$, we define $D: \Zd \to [0,1]$ by
\begin{align}
\lbeq{Ddef}
D(x)=
 \begin{cases}
 \dfrac{h(x/L)}{\sum_{y\in\Zd\setminus\{0\}}h(y/L)}\quad&(x\ne 0)\\
 0& (x=0).
 \end{cases}
\end{align}
It follows that, for any $n\in\N$,
\begin{align}
\lbeq{Dbds}
\|D\|_\infty=O(L^{-d}),&&
\sup_{x\in \Zd}|x|^nD(x)=O(L^{n-d}),&&
\sum_{x\in \Zd}|x|^nD(x)=O(L^n).
\end{align}
We use big-$O$ notation in the standard way:  $f(x)=O(g(x))$ means that
there is a constant $M$ such that $|f(x)| \le Mg(x)$.

We focus in this paper on spread-out models because we rely for diagrammatic
estimates on the results of the inductive approach to the lace expansion \cite{HS02,HHS08},
which has only been developed for spread-out models.  The spread-out models also have the
advantage that they can be analysed using our methods for all dimensions above the
upper critical dimension ($4$ for oriented percolation and the contact process, $8$
for lattice trees).  We believe that our results could be extended, with more effort,
to nearest-neighbour models in sufficiently high dimensions, but we do not pursue that
here.

\subsection{Oriented percolation}

We begin with oriented percolation,
which also serves as a discretisation of the contact process.
We will use this discretisation to analyse the contact process.

Spread-out oriented percolation is defined on
the graph with vertex set $\Zd \times \{0,1,2,\ldots\}$ and
directed bonds $((x,n),(y,n+1))$, for $x,y \in
\Zd$ and $n \ge 0$.
To the directed bonds $((x,n),(y,n+1))$, we associate
independent random variables taking the value 1 with probability
$pD(y-x)$ and 0 with probability $1-pD(y-x)$.
We say a bond is \emph{occupied}  when its random variable takes
the value 1, and \emph{vacant} when its random variable is 0.
The parameter $p \in [0,\|D\|_\infty^{-1}]$ is the expected
number of occupied bonds per vertex (it is \emph{not} a probability).
The joint probability distribution of the bond variables
is $\Pbold_p$, with expectation denoted
$\Ebold_p$.

We say that $(x,n)$ is \emph{connected}  to $(y,m)$, and write $(x,n)
\longrightarrow (y,m)$, if there is an oriented path from $(x,n)$ to
$(y,m)$ consisting of occupied bonds, or if $(y,m)=(x,n)$.
This requires that $n\leq m$.
Let $C(x,n)$ denote the set of sites
$(y,m)$ such that $(x,n) \longrightarrow (y,m)$; its cardinality is denoted
by $|C(x,n)|$.
The \emph{critical point} $p_c$ is defined to be the supremum of the set of
$p\in [0,\|D\|_\infty^{-1}]$ for which
$\Ebold_p|C(0,0)|<\infty$.   It is known that $p_c=1+O(L^{-d})$ (and more)
\cite{HS05y}.  Let
\begin{equation}
    \tau_n(x) = \Pbold_{p_c}((0,0) \longrightarrow (x,n)).
\end{equation}
Our main result for oriented percolation is the following theorem.

\begin{theorem}
\label{thm:op}
Let $d>4$.  There is an $L_0=L_0(d)$ such that for $L \ge L_0$,
and for any $s\ge 0$, there is a constant
$c_s=c_s(L)$ such that
\begin{equation}
\lbeq{op}
    \sum_{x\in \Zd}|x|^{s}\tau_n(x) \le c_s n^{s/2}.
\end{equation}
\end{theorem}

Stronger results than Theorem~\ref{thm:op} have been proved
previously.  For $s=0,2$, the existence of the limit
\begin{equation}
    A_s = \lim_{n\to\infty}n^{-s/2}\sum_{x\in \Zd}|x|^{s}\tau_n(x)
\end{equation}
was proved in \cite[Theorem~3]{NY95} and also in \cite[Theorem~1.1]{HS03b},
and this was extended to general $s \ge 0$ in \cite{CS11}
in the case where $|x|^s$ is replaced by $|x_1|^s$.

\subsection{Contact process}

The contact process is a continuous-time Markov process with
state space $\{0,1\}^\Zd$, with $d \ge 1$.
The state of
the contact process is determined by a variable
$\xi_x \in \{0,1\}$, for each $x\in\Zd$.  When $\xi_x=0$,
then $x$ is \emph{healthy}, and when $\xi_x=1$, then
$x$ is \emph{infected}.  An infected particle spontaneously
becomes healthy at rate 1, and, given $p>0$, a healthy particle at $x$ becomes
infected
at rate $p\sum_{y\in\Zd}\xi_yD(x-y)$.

We assume that at time zero there is a single infected individual at the origin,
with all others healthy.  Let $C_t$ denote the set of infected particles at
time $t \ge 0$.
The \emph{susceptibility} is defined by
\begin{equation}
    \chi(p) = \sum_{x\in\Zd} \int_0^\infty \Pbold_{p}(x \in C_t)\,\text{d}t.
\end{equation}
The \emph{critical point} is defined by
$p_c=\sup\{p : \chi(p)<\infty\}$.
It is known that $p_c=1+O(L^{-d})$ (and more) \cite{HS05y}.
The following theorem is our main result for the contact process.

\begin{theorem}\label{thm:cp}
Let $d>4$.  There is an $L_0=L_0(d)$ such that for $L \ge L_0$,
and for any $s\ge 0$, there is a constant
$c_s=c_s(L)$ such that
\begin{equation}
\lbeq{cp}
    \sum_{x\in \Zd}|x|^{s}\Pbold_{p_c} (x \in C_t) \le c_s
    \begin{cases}
    t^{s/2} & (t \ge 1 )
    \\
    1 & (t < 1 ).
    \end{cases}
\end{equation}
\end{theorem}

The dichotomy in \refeq{cp} does not arise for
integer-time models such as oriented percolation and lattice trees.
For $s=0,2$, the existence of the limit
\begin{equation}
    A_s = \lim_{t\to\infty}t^{-s/2}\sum_{x\in \Zd}|x|^{s}\Pbold_{p_c} (x \in C_t)
\end{equation}
was proved in \cite[Theorem~1.1]{HS04z}.

It is well known that the contact process can be approximated by an
oriented percolation model (see, e.g., \cite{BW98,BG91,HS10,HS04z,Saka01}).
For this, we replace the time interval $[0,\infty)$ by
$\varepsilon\Z_+ \equiv\{0,\varepsilon,2\varepsilon,3\varepsilon,\dots\}$,
and define an oriented percolation model
on $\Zd \times \varepsilon \Z_+$, as follows.  Bonds have the
form $((x,t),(y,t+\varepsilon))$ with $t \in \varepsilon \Z_+$
and $x,y \in\Zd$.  A bond is occupied
with probability $1-\varepsilon$ if $x=y$,
and with probability $\varepsilon pD(y-x)$
if $x\ne y$.  Bonds with $x=y$ are called \emph{temporal}
and bonds with $x \neq y$ are called \emph{spatial}.
Let $\Pbold^\varepsilon_{p}$ denote the probability measure
for the  oriented percolation model
on $\Zd \times \varepsilon \Z_+$.  Then
$\Pbold^\varepsilon_{p}$
converges weakly as $\varepsilon \to 0^+$ to the contact process
measure $\Pbold_{p}$, and the critical value $p_c^\varepsilon$
of the discretised model converges to the critical value $p_c$ of the
contact process \cite{BG91,Saka01}
(see also \cite{Wu95}).  In particular,
\begin{align}
\lbeq{cptoop}
\lim_{\varepsilon\to0^+}
\Pbold_{p_c^\varepsilon}^\varepsilon\big((0,0)\longrightarrow(x,
\lfloor t/\varepsilon \rfloor \varepsilon)\big)
 =
 \Pbold_{p_c}(x\in C_t),
\end{align}
where $\lfloor \cdot \rfloor$ denotes the floor function.
In order to deal with all models simultaneously, we adopt the notational convenience
\begin{align}
\tau_n(x)=\Pbold_{p_c^\varepsilon}^\varepsilon\big((0,0)
 \longrightarrow(x,n\varepsilon)\big).
\end{align}

\subsection{Lattice trees}
\label{sec:lt-intro}

For lattice trees, we assume that $h(x)=0$ if $\|x\|_\infty >\frac 12 $, so that $D$ of
\refeq{Ddef} is supported on $[-\frac L2, \frac L2]^d$.
Let $B$ be the set of bonds $\{x,y\}$ in $\Zd$ with $0< \|x-y\|_\infty  \le L$.
A lattice tree is a finite connected set of bonds in $B$ with no
cycles.  Let ${\cal T}_N$ denote the set of $N$-bond lattice trees containing
the origin $0$, let $B(T)$ denote the set of bonds in $T \in {\cal T}_N$,
and let
\begin{equation}
    t_N^{(1)} = \sum_{T\in {\cal T}_N} \prod_{\{x,y\}\in B(T)} D(x-y).
\end{equation}
For $x \in \Zd$, let ${\cal T}_{N,n}(x)$ denote the set of lattice trees $T\in {\cal T}_N$
which contain $x$ and for which the unique path in $T$ connecting $0$ and $x$
consists of $n$ bonds.
Let
\begin{equation}
    t_N^{(2)}(x;n) = \sum_{T\in {\cal T}_{N,n}(x)} \prod_{\{x,y\}\in B(T)} D(x-y).
\end{equation}
A standard subadditivity argument implies that there exists $p_c>0$
such that the 1-point function
$g_p = \sum_{N=0}^\infty t_N^{(1)}p^N$
has radius of convergence $p_c$.

Critical exponents and the scaling limit of lattice trees in dimensions $d>8$
are discussed in, e.g., \cite{DS98,Slad06,Holm08}.
Here, we rely on results from \cite{Holm08}.
Our assumption that $D$ has finite range is
to conform with \cite{Holm08}; we expect
that this restriction is actually unnecessary.
It is known that $p_c=1+O(L^{-d})$ and $1 \le g_{p_c} \le 4$,
if $d>8$ and if $L$ is large enough \cite{HS90b}.
Let
\begin{equation}
    \tau_n(x) = \sum_{N=0}^\infty t_N^{(2)}(x;n)p_c^N,
\end{equation}
which is finite since $t_N^{(2)}(x;n) \le t_N^{(1)}$.
The following theorem is our main result for lattice trees.

\begin{theorem}
\label{thm:lt}
Let $d>8$.  There is an $L_0=L_0(d)$ such that for $L \ge L_0$,
and for any $s\ge 0$, there is a constant
$c_s=c_s(L)$ such that
\begin{equation}
\lbeq{LT-smom}
    \sum_{x\in\Zd} |x|^s \tau_n(x) \le c_s n^{s/2}.
\end{equation}
\end{theorem}

For $s=0,2$, the existence of the limit
\begin{equation}
    A_s = \lim_{n\to\infty}n^{-s/2}\sum_{x\in \Zd}|x|^{s}\tau_n(x)
\end{equation}
was proved in  \cite[Theorem~3.7]{Holm08}, under the assumptions of Theorem~\ref{thm:lt}.
The cases $s=0,2,4,6$ of \refeq{LT-smom} are proved in detail in \cite{HP19}.
Our method of proof is
simpler than that of \cite{HP19}, and it applies to all $s \ge 0$.

\subsection{Asymptotic formulas}
\label{sec:af}

According to \cite[Remark~2.8]{HP20}, the conclusions of Holmes and Perkins
include the promotion of the upper bounds on the spatial moments,
once they have been established, to precise asymptotic formulas.
For oriented percolation and lattice trees obeying the hypotheses of Theorems~\ref{thm:op}
and \ref{thm:lt}, this implies existence of the limits
\begin{equation}
    A_s = \lim_{n\to\infty} n^{-s/2}\sum_{x\in\Z^d} |x|^s \tau_n(x)
    \qquad \text{for all $s \ge 0$.}
\end{equation}
Similarly, for the contact process obeying the hypotheses of Theorem~\ref{thm:cp},
we have existence of the
limits
\begin{equation}
    A_s = \lim_{t\to\infty}t^{-s/2}\sum_{x\in \Zd}|x|^{s}\Pbold_{p_c} (x \in C_t)
    \qquad \text{for all $s \ge 0$.}
\end{equation}
The limits $A_s$ of course depend on the model.  However, in each case, they can
be expressed in terms of $A_0$, $A_2$ and the dimension $d$ as
\begin{equation}
\lbeq{amplitude}
    A_s
    = A_0 (A_2/(A_0 d))^{s/2} E[|Z|^s]
    \qquad \text{for all $s \ge 0$,}
\end{equation}
where $Z$ is a $d$-dimensional vector whose components are each standard normal
random variables.
For the contact process and oriented percolation,
the amplitudes $A_0,A_2$ obey $A_0=1+O(L^{-d})$ and $A_2 = \sigma^2[1+O(L^{-d})]$,
where $\sigma^2$ is the variance of the probability distribution $D$ of \refeq{Ddef}
\cite{HS03b,HS04z}.  For lattice trees, $A_0$ is within a factor
$[1+O(L^{-d/2})]$ of the critical 1-point function, and $A_2=A_0\sigma^2[1+O(L^{-d/2})]$
\cite[Theorem~3.7]{Holm08}.

If we instead consider moments where $|x|^s$ is replaced by
$|x_1|^s$, then, as pointed out in \cite{HP20},
the formula \refeq{amplitude} holds instead for $Z$ a 1-dimensional
standard normal, which obeys $E[|Z|^s] = \Gamma(s+1)/[\Gamma(\frac s2 + 1)2^{s/2}]$.
Using the notation for oriented percolation and lattice trees, this gives the limiting value
\begin{equation}
\lbeq{HPA1}
    \lim_{n\to\infty} n^{-s/2}\sum_{x\in\Z^d} |x_1|^s \tau_n(x)
    =
    A_0 \left(\frac{A_2}{2dA_0}\right)^{s/2}\frac{\Gamma(s+1)}{\Gamma(\frac s2 + 1)},
\end{equation}
and similarly for the contact process.
To compare this with the results of \cite{CS11} for oriented percolation, we
recall from \cite[(1.6)]{CS11} and \cite[(1.9)]{CS11} that there exist positive
$C_\text{I},C_\text{II}$ such that
\begin{align}
\lim_{n\to\infty}\sum_{x\in\Z^d}\tau_n(x)=C_\text{I},\quad
\lim_{n\to\infty}\frac{\sum_{x\in\Z^d}|x_1|^s\tau_n(x)}{n^{s/2}\sum_{x\in\Z^d}\tau_n(x)}
 =(C_\text{II}v_\infty)^{s/2} \frac{\Gamma(s+1)}{\Gamma(\frac{s}2+1)},
\end{align}
where $v_\infty=\sigma^2/(2d)$ (see \cite[(1.2)]{CS11}).
This is consistent with \refeq{HPA1}, with
$A_0=C_\text{I}$ and $A_2=C_{\text{I}} C_\text{II}\sigma^2$.

\section{Sufficient condition on generating function}
\label{sec:suffcond}

For any one of the three models under consideration, let
\begin{equation}
\lbeq{tzdef}
  t_z(x) =  \sum_{n=0}^{\infty} \tau_{n}(x) z^n.
\end{equation}
The Fourier transform of an absolutely summable function $f : \Zd \to \R$ is defined by
\begin{equation}
    \hat{f}(k)
    = \sum_{x \in \Zd} f(x)
    e^{ik \cdot x}
    \qquad(k=(k_1,\ldots,k_d) \in [-\pi,\pi]^d).
\end{equation}
We use the Fourier--Laplace transform
\begin{equation}
    \hat{t}_z(k) = \sum_{n=0}^\infty \hat\tau_n(k) z^n .
\end{equation}
Let $\Delta = \sum_{i=1}^d \ddp{^2}{k_i^2}$.  Then, for
$|z|<1$ and $r\in \N$,
\begin{align}
    \Delta^r \hat t_z(0)
    =
    \sum_{n=0}^\infty d^{(r)}_n z^n
    & \quad \text{with} \quad
    d^{(r)}_n
    = (-1)^r  \sum_{x\in \Zd}|x|^{2r} \tau_n(x) .
\end{align}

A key element of our proof is the following minor extension of \cite[Lemma~3.2(i)]{DS98},
which is a kind of Tauberian theorem (see also \cite{FO90}).

\begin{lemma}\label{lem:DS}
Suppose that, for $N \ge 1$ and for $u_j \ge 1$, $v_j \ge 0$, the power series
$f(z) = \sum_{n=0}^\infty a_nz^n$ has radius of convergence at least $1$ and obeys
\begin{equation}
  \left| f(z) \right| \le \sum_{j=1}^N \frac{C_j}{|1-z|^{u_j} (1-|z|)^{v_j}}
  \qquad
  (|z|<1).
\end{equation}
Then there is a constant $c$ depending only on $u_1,\ldots,u_N$ such that
\begin{equation}
  \left| a_n \right| \le
  c\sum_{j=1}^N C_j n^{v_j} \times
  \begin{cases}
  n^{u_j-1} & (u_j>1)
  \\
   \log n & (u_j=1).
  \end{cases}
\end{equation}
\end{lemma}

\begin{proof}
Let $\Gamma_n$ denote the circle of radius $r_n= 1-\frac{1}{n}$
centred at the origin and with counterclockwise orientation.
By the Cauchy Integral Formula,
\begin{align}
\lbeq{RL199}
	a_n & = \frac{1}{2\pi i} \oint_{\Gamma_n}  f(z) \frac{dz}{z^{n+1}}
    =
    \frac{1}{r_n^n} \int_{-\pi}^\pi f(r_ne^{i\theta}) e^{-in\theta} \frac{d\theta}{2\pi}.
\end{align}
Therefore, by hypothesis, and since $(1-1/n)^{-n}$ is bounded above by a universal
constant $K$,
\begin{equation}
	|a_n|
    \leq
    \frac{K}{2\pi} \sum_{j=1}^N C_j n^{v_j}
    \int_{-\pi}^\pi |1-r_n e^{i\theta}|^{-u_j} d\theta.	
\end{equation}
The integral on the right-hand side is analysed in the proof of
\cite[Lemma~3.2(i)]{DS98}, where it is shown to be bounded by $n^{u_j-1}$ if $u_j>1$,
or by $\log n$ if $u_j=1$, times a constant depending only on $u_j$.
This completes the proof.
\end{proof}

We will prove the following proposition for the discretised contact
process.  The proposition also applies for oriented percolation for $d>4$ and
lattice trees for $d>8$, with the parameter $\varepsilon$ given simply by
$\varepsilon=1$.

\begin{prop}\label{prop:Deltart}
Let $d>4$, $p=p_c^\varepsilon$, and $r \in \N$.
There is an $\varepsilon$-independent $L_0>0$ such that for any $L \ge L_0$
there is an $\varepsilon$-independent constant $C_{2r}=C_{2r}(L)$ such that
\begin{align}\lbeq{IRr}
|\Delta^r\hat t_z(0)|\le\frac{C_{2r}\varepsilon}{|1-z|^2}\sum_{j=0}^{r-1}
 \frac{ \varepsilon^j }{(1-|z|)^j}
 \qquad (|z|<1).
\end{align}
\end{prop}

By Lemma~\ref{lem:DS}, Proposition~\ref{prop:Deltart}
implies that there is an $\varepsilon$-independent constant $c'_{2r}$
such that, for all $r,n\in\N$,
\begin{align}
\sum_{x\in\Zd}|x|^{2r}\tau_n(x)\le c'_{2r}\sum_{j=1}^r (n\varepsilon)^j.
\end{align}
When $\varepsilon =1$ this gives a bound $c_{2r} n^r$ and proves Theorems~\ref{thm:op}
and \ref{thm:lt} for even integer powers.
For the contact process, the bound becomes
\begin{align}
\lbeq{cp2r}
\sum_{x\in\Zd}|x|^{2r}
\Pbold_{p_c^\varepsilon}^\varepsilon ((0,0) \to (x,n\varepsilon))
\le c'_{2r}r\big((n\varepsilon)\vee(n\varepsilon)^r\big).
\end{align}
By \refeq{cptoop}, this implies the conclusion of Theorem~\ref{thm:cp}
for even integer powers.

In fact, it suffices to prove Theorems~\ref{thm:op} and \ref{thm:lt}
for $s$ a non-negative even
integer, since the general case then follows by H\"older's inequality
(using the known results for the zeroth moment).
In detail, if $s$ is not an even integer then let $r$ be the smallest
integer such
that $s<2r$.  Let $a=2r/s>1$ and define $b$ by
$\frac {1}{a} + \frac {1}{b} =1$.
Then $sa=2r$ is an even integer and, for oriented percolation and
lattice trees ($\varepsilon=1$),
\begin{align}\lbeq{Holder1}
\sum_{x\in \Zd}|x|^{s}\tau_n(x)
& = \sum_{x\in \Zd}|x|^{s}\tau_n(x)^{1/a}\tau_n(x)^{1/b}\nnb
&\le \bigg(\sum_{x\in \Zd}|x|^{sa}\tau_n(x)\bigg)^{1/a}\bigg(\sum_{x\in \Zd}
 \tau_n(x)\bigg)^{1/b}
 \nnb & \le
 (c_{2r}n^r)^{1/a} c_0^{1/b}
 = c_s n^{s/2}.
\end{align}
For the contact process, the above argument gives instead,
by using \refeq{cp2r} for the $(2r)^\text{th}$ moment, the upper bound
(with $t=n\varepsilon$)
\begin{equation}
\big(c'_{2r}r(t\vee t^{r})\big)^{1/a}c_0^{1/b} = c_s(t^{s/2r}\vee t^{s/2}).
\end{equation}
For $t \le 1$, the right-hand side is $O(1)$, consistent with \refeq{cp}.
For $t \ge 1$, since $s/2r \le s/2$ the right-hand side is $O(t^{s/2})$, which is
again consistent with \refeq{cp}.
Thus, for Theorem~\ref{thm:cp} it also suffices
to consider even integer powers $s$.

\section{Oriented percolation: proof of Theorems~\ref{thm:op}--\ref{thm:cp}}

We fix $d>4$, $L$ sufficiently large, and $p= p_c^\varepsilon$ throughout this section.
We are generally not concerned with the $L$-dependence of constants,
and usually allow them to depend on $L$.
As discussed above, we can restrict to even integer powers $s$, and since the cases
$s=0,2$ are already well established in previous papers, we can restrict
attention to even integers $s \ge 4$.

\subsection{Lace expansion for oriented percolation}

The proof uses the lace expansion, which provides a formula for the
coefficients $\pi_n(x)$ in the convolution equation
\begin{equation}
\lbeq{taunx}
    \tau_{n+1}(x) =  (q * \tau_{n})(x) +
    \sum_{m=2}^{n}
    ( \pi_m *q * \tau_{n-m})(x)
    + \pi_{n+1}(x)
    \quad\quad (n \geq 0),
\end{equation}
with the convention that an empty sum is zero.
Here $(f*g)(x) = \sum_{y \in \Zd}f(y)g(x-y)$, $\pi_0(x)=\pi_1(x)=0$ for all $x$, and
\begin{align}
\lbeq{qdef}
q(x)=(1-\varepsilon)\delta_{0,x}+\varepsilon p_c^\varepsilon D(x).
\end{align}
By \cite[(2.30) \& (2.34)]{HS04z} (and \cite[(1.12)]{HS03b} for
$\varepsilon=1$), $ p_c^\varepsilon=1+O(L^{-d})$.
There are three different versions of the lace expansion for oriented
percolation, which provide different representations for
$\pi_n(x)$ \cite{NY93,HS90a,Saka01}.  See \cite{Slad06} for a discussion
of these three representations.  We require few details here.

The following proposition is proved in \cite{HS04z} (see also \cite{HS03b,NY95}
for $\varepsilon=1$,
and note that the amplitudes $A^{(\varepsilon)}, v^{(\varepsilon)}$ in
\cite[Propositions~2.1 \& 2.3]{HS04z} are controlled in \cite[Proposition~2.6]{HS04z}).
We will extend Proposition~\ref{prop:HSbds} to arbitrary integers $r \ge 0$.

\begin{prop}
\label{prop:HSbds}
{\rm \cite[Propositions~2.1 \& 2.3]{HS04z}}
Let $d>4$ and $p= p_c^\varepsilon$.
There is an $L_0>0$ and a finite $C$, both independent of
$\varepsilon$, such that for $L \ge L_0$ and for $n \ge 0$,
\begin{align}
\sum_{x\in\Zd}|x|^{2r}\tau_n(x)&\le C(L^2n\varepsilon)^r &(r=0,1),
 \lbeq{taunbd1}\\
\sup_{x\in\Zd}|x|^{2r}\tau_n(x)&\le(1-\varepsilon)^n+C\big(L^2(1+n\varepsilon)
 \big)^{r-d/2} &(r=0,1),\lbeq{taunbd2}\\
\sum_{x\in\Zd} |x|^{2r}|\pi_n(x) |&\le\varepsilon^2C\big(L^2(1+n\varepsilon)
 \big)^{r-d/2} &(r=0,1,2).\lbeq{pinbd}
\end{align}
\end{prop}

We use the Fourier--Laplace transforms, for complex $z$ and for $k \in [-\pi,\pi]^d$,
\begin{align}
&\hat{t}_z(k) = \sum_{n=0}^\infty \hat\tau_n(k) z^n,&
&\hat{\Pi}_z(k) = \sum_{n=2}^\infty \hat\pi_n(k) z^n.
\end{align}
We will also have use for the definitions
\begin{align}
&\hat q(k) = 1-\varepsilon+\varepsilon p_c^\varepsilon \hat D(k),&
&\hat\Phi_z(k) = z\hat q(k)\big(1+\hat\Pi_z(k)\big).\lbeq{Phiqdef}
\end{align}
By \refeq{taunbd1} with $r=0$, the power series $\hat{t}_z(k)$ converges
in the open unit disk $|z|<1$ in the complex plane.
It is known that
the limit $A=\lim_{n\to\infty} \hat\tau_n(0) =1+O(L^{-d})$
exists for $d>4$ and $L$ sufficiently
large \cite{HS03b,HS04z}.
Since, by the Markov
property, $\hat\tau_{m+n}(0) \le \hat\tau_m(0)\hat\tau_n(0)$, by Fekete's Lemma the limit
$\mu= \lim_{n\to\infty}\hat\tau_n(0)^{1/n}$ also exists and $\hat\tau_n(0) \ge \mu^n$ for all $n$.
It follows that $\mu$ must be $1$  and that $\hat t_z(0) \ge \sum_{n=0}^\infty z^n = (1-z)^{-1}$
for $z\in [0,1)$.
In particular, $\hat{t}_1(0)=\infty$.
On the other hand, the following corollary to Proposition~\ref{prop:HSbds} shows that
$\Pihat_z$ and some of its derivatives remain bounded on the
\emph{closed} disk $|z|\le 1$.

\begin{cor}\label{cor:Pibds}
Under the hypotheses of Proposition~\ref{prop:HSbds},
uniformly in $|z|\le 1$,
\begin{align}
|\hat\Pi_z(0)|& \le O(L^{-d}\varepsilon) ,\lbeq{Pi}\\
|\partial_z\hat\Pi_z(0)|&\le O(L^{-d})\lbeq{dzPi},\\
|\Delta\hat\Pi_z(0)|&\le O(L^{2-d}\varepsilon) ,
 \lbeq{DelPi}
\\
\lbeq{DelPhi}
|\Delta\hat\Phi_z(0)| & \le
O(L^2\varepsilon).
\end{align}
\end{cor}

\begin{proof}
Let $|z|\le 1$.
By \refeq{pinbd} with $r=0,1$,
\begin{align}
|\hat\Pi_z(0)|&\le\sum_{n=2}^\infty \sum_{x\in\Zd}|\pi_n(x)|\le O(L^{-d})~\varepsilon^2
 \sum_{n=2}^\infty(1+n\varepsilon)^{-d/2}\le O(L^{-d}\varepsilon),
 \\
|\partial_z\hat\Pi_z(0)|&\le\sum_{n=2}^\infty\sum_{x\in\Zd} n\,|\pi_n(x)|\le O(L^{-d})~
 \varepsilon\sum_{n=2}^\infty(1+n\varepsilon)^{1-d/2}\le O(L^{-d}),
 \\
|\Delta\hat\Pi_z(0)|&\le\sum_{n=2}^\infty\sum_{x\in\Zd}|x|^2|\pi_n(x)|\le O(L^{2-d})~
 \varepsilon^2\sum_{n=2}^\infty(1+n\varepsilon)^{1-d/2}\le O(L^{2-d}\varepsilon)  .
\end{align}
Also, by \refeq{Dbds},
\begin{align}
|\Delta\hat\Phi_z(0)|=|z|\Big|\hat q(0)\,\Delta\hat\Pi_z(0)+\big(1+\hat\Pi_z(0)
 \big)\varepsilon  p_c^\varepsilon \Delta\hat D(0)\Big|=O(L^2\varepsilon) .
\end{align}
This completes the proof.
\end{proof}

The Fourier-Laplace transform converts the
spatio-temporal convolutions in \refeq{taunx} into products, which leads to the following lemma.

\begin{lemma}
For $|z|<1$ and $k \in [-\pi,\pi]^d$,
\begin{align}\lbeq{tidentity}
\hat t_z(k)=1+\hat\Pi_z(k)+
 \hat\Phi_z(k)
 \hat t_z(k)=\frac{1+\hat \Pi_z(k)}
 {1-\hat\Phi_z(k)}.
\end{align}
\end{lemma}

\begin{proof}
We take the Fourier transform of \refeq{taunx}, multiply by $z^{n+1}$,
and sum over $n \ge 0$, to obtain
\begin{equation}
    \sum_{n=0}^\infty \hat\tau_{n+1}z^{n+1}
    = z\hat q \sum_{n=0}^\infty\hat\tau_n z^n
    + z\hat q \sum_{n=0}^\infty\sum_{m=2}^{n}\hat\pi_{m}z^m \hat\tau_{n-m}z^{n-m}
    + \sum_{n=0}^\infty\hat{\pi}_{n+1}z^{n+1}
    .
\end{equation}
The left-hand side is $\hat t_z -1$, and the first and last terms on the right-hand side
are respectively $z\hat q \hat t_z$ and $\hat\Pi_z$ (recall that $\pi_n=0$ for $n=0,1$).
After interchange of summation, the middle term on the right-hand side becomes
$z\hat q \hat t_z \hat \Pi_z$.  Thus we have
\begin{equation}
    \hat t_z = 1 + z\hat q (1+ \hat\Pi_z)\hat t_z + \hat \Pi_z,
\end{equation}
as required.

Finally, there can be no division by zero on the right-hand side of
\refeq{tidentity}, since $\hat{t}_z(k)$ converges for $|z|<1$,
and the numerator is close to $1$ by \refeq{Pi}.
\end{proof}

Since $\lim_{z\uparrow1}\hat t_z(0)=\infty$
(as argued above Corollary~\ref{cor:Pibds})
and since $\hat\Pi_1(0)=O(L^{-d}\varepsilon)$,
it follows from \refeq{tidentity} that
\begin{equation}\lbeq{F10}
\hat\Phi_1(0)=1.
\end{equation}
By \refeq{tidentity} and \refeq{Phiqdef}, we can therefore rewrite $\hat t_z(0)$ as
\begin{align}
\hat t_z(0)=\frac{1+\hat \Pi_z(0)}{\hat\Phi_1(0)-\hat\Phi_z(0)}
 =\frac{1+\hat \Pi_z(0)}
 {\hat q(0)\big((1-z)(1+\hat\Pi_1(0))+z(\hat \Pi_1(0)-\hat\Pi_z(0))\big)}.
\end{align}
It then follows from \refeq{Pi}--\refeq{dzPi} that
\begin{align}\lbeq{IR0}
|\hat t_z(0)|
  \le \frac{O(1)}{|1-z|}\qquad (|z|<1).
\end{align}

To bound the second moment $\Delta\hat t_z(0)$, we differentiate \refeq{tidentity}
using the quotient rule and apply \refeq{DelPi}--\refeq{DelPhi}.
Since $\nabla \hat D(0)=0$ due to the symmetry of $D$, and similarly for $\hat\Pi$
and $\hat\Phi$,
this
gives (with $L$-dependent constant)
\begin{align}\lbeq{IR1}
|\Delta\hat t_z(0)|&=\bigg|\frac{1}{1-\hat\Phi_z(0)}\Big(\Delta\hat
 \Pi_z(0)+\hat t_z(0)\Delta\hat\Phi_z(0)\Big)\bigg|\nnb
&=\frac{O(1)}{|1-z|}\bigg(\varepsilon+\frac{\varepsilon}{|1-z|}\bigg)
 =\frac{O(\varepsilon)}{|1-z|^2}\qquad \qquad (|z|<1).
\end{align}
Note that the bounds \refeq{IR0}--\refeq{IR1} are better than those
naively obtained by multiplying $|z|^n$ on both sides of \refeq{taunbd1} and
then summing over nonnegative integers $n$, which results in
inverse powers of $(1-|z|)$ instead of $|1-z|$.

\subsection{Induction}

The proof of Proposition~\ref{prop:Deltart} is by induction on $r\in\N$.
The case $r=1$ is provided by \refeq{IR1}.  To advance the induction,
we first observe that, by H\"older's inequality,
\begin{align}
|x|^{2r}
 =\bigg(\sum_{j=1}^dx_j^2\bigg)^r
 \le\Bigg(\bigg(\sum_{j=1}^d1\bigg)^{1-1/r}\bigg(\sum_{j=1}^dx_j^{2r}
  \bigg)^{1/r}\Bigg)^r
 =d^{r-1}\sum_{j=1}^dx_j^{2r}.
\end{align}
Then, for a $\Zd$-symmetric nonnegative function $f$ on $\Zd$,
we have
\begin{align}\lbeq{equivalence}
\sum_{x\in\Zd} x_1^{2r}f(x)\le\sum_{x\in\Zd}|x|^{2r}f(x)\le d^{r-1}\sum_{x\in\Zd}\sum_{j=1}^d
 x_j^{2r}f(x)=d^r\sum_{x\in\Zd}x_1^{2r}f(x).
\end{align}
Therefore, to prove Proposition~\ref{prop:Deltart}, it suffices to show
\refeq{IRr} for a single component, i.e.,
\begin{align}\lbeq{IRr-1comp}
|\partial_1^{2r}\hat t_z(0)|\le\frac{O(\varepsilon)}{|1-z|^2}\sum_{j=0}^{r-1}
 \frac{ \varepsilon^j }{(1-|z|)^j}
 \qquad (|z|<1),
\end{align}
where $\partial_1$ is an abbreviation for $\frac\partial{\partial k_1}$.

By applying Leibniz's rule to the first equality in \refeq{tidentity} and using
the spatial symmetry, we obtain
\begin{align}\lbeq{Deltar}
\partial_1^{2r}\hat t_z(0)&=\partial_1^{2r}\hat\Pi_z(0)+\sum_{j=0}^r
 \binom{2r}{2j}\partial_1^{2j}\hat\Phi_z(0)\,\partial_1^{2r-2j}\hat t_z(0)\nnb
&=\frac1{1-\hat\Phi_z(0)}
 \bigg(\partial_1^{2r}\hat\Pi_z(0)+\sum_{j=1}^r\binom{2r}{2j}\partial_1^{2j}
 \hat\Phi_z(0)\,\partial_1^{2r-2j}\hat t_z(0)\bigg)
 \qquad (r\in\N).
\end{align}
(This reproduces the first equality of \refeq{IR1} when $r=1$.)
The right-hand side of \refeq{Deltar} only involves lower-order
derivatives of $\hat t$ than the left-hand side.
This opens up the possibility of an
inductive proof, though we must deal with the fact that the right-hand side
does involve $\partial_1^{2r}\hat \Pi_z$.
The idea is similar to that of \cite{CS11}, where for oriented percolation
an asymptotic formula for
$\sum_{x\in\Zd}|x_1|^st_n(x)$ is derived for any $s>0$.
We only prove upper bounds in this paper, and can be
less careful in dealing with the recursion relation \refeq{Deltar}.

The next lemma shows that $\ell_1$ estimates on $|x|^q\tau_n(x)$ imply
corresponding $\ell_\infty$ estimates.

\begin{lemma}\label{lem:ellnorms}
Assume the same setting as Proposition~\ref{prop:Deltart}, and
let $q\in\N$.
Suppose there is an $\varepsilon$-independent constant $C$ such that
$\sum_{x\in\Zd} |x|^q\tau_n(x) \le C(1+n\varepsilon)^{q/2}$ holds for all
$n\ge0$.  Then there is an $\varepsilon$-independent constant $C'$
such that $\sup_{x\in\Zd} |x|^q \tau_n(x) \le C'(1+n\varepsilon)^{(q-d)/2}$ also
holds for all $n\ge0$.
\end{lemma}

\begin{proof}
We write $m=\lfloor n/2\rfloor$.  Since $(a+b)^q \le 2^q(a^q+b^q)$ for
any $a,b>0$, and since $\tau_n(x) \le \sum_y \tau_m(y)\tau_{n-m}(x-y)$,
we have
\begin{align}
|x|^q \tau_n(x)
&\le2^q\|\tau_{n-m}\|_\infty\sum_y |y|^q\tau_m(y)
 +2^q\|\tau_m\|_\infty\sum_y|x-y|^q\tau_{n-m}(x-y).
\end{align}
By hypothesis, each of the two sums is bounded by $C(1+n\varepsilon)^{q/2}$.
By \refeq{taunbd2} for $r=0$, each $\ell_\infty$ norm is bounded above
by a multiple of $(1+n\varepsilon)^{-d/2}$ (we use
$(1-\varepsilon)^n \le e^{-n\varepsilon} \le c(1+n\varepsilon)^{-d/2}$
with $c$ independent of $\varepsilon$).
This completes the proof.
\end{proof}

The next lemma is a key step in advancing the induction.  It shows
that a bound on the $(2r)^{\rm th}$ moment of $\tau_n$ implies
a bound on the $(2r+2)^{\rm th}$ moment of $\pi_n$.
The proof uses standard diagrammatic estimates on $\pi_n$, which are
explained in detail in \cite[(4.26)]{HS04z} (and \cite[(4.10)]{HS03b} for $\varepsilon=1$).
Figure~\ref{fig:pi1} illustrates the diagrams for $\pi_n^{\scriptscriptstyle(1)}(x)$,
which is one contribution to $\pi_n(x)$.
The oriented percolation case $\varepsilon=1$ (left figure) needs
less care, as compared to the contact process case $\varepsilon<1$ (right
figure).  For the latter, correct factors of $\varepsilon$ must be extracted.
Those factors can be obtained by noting that,
at each intersection of two bond-disjoint paths, at least one
must use a spatial bond (red arrows) to leave or enter that point.
Summation over the temporal location of the two middle red bonds consumes their
$\varepsilon$ factors to form a Riemann sum, leaving the two $\varepsilon$ factors
from the top and bottom red bonds; these are responsible for the factor $\varepsilon^2$
in the bound \refeq{pinbd} on $\pi_n(x)$.
In the diagrammatic estimates of \cite{HS03b,HS04z}, it was also necessary
to extract an inverse power of $L$ for each diagram loop in order to guarantee convergence
of the sum over diagrams with increasing complexity.  Once that work has been done,
it is not necessary for us to consider $L$-dependence here, since our operations modify
the estimates essentially only in a single loop.

\begin{figure}
\begin{align*}
\includegraphics[scale=.40]{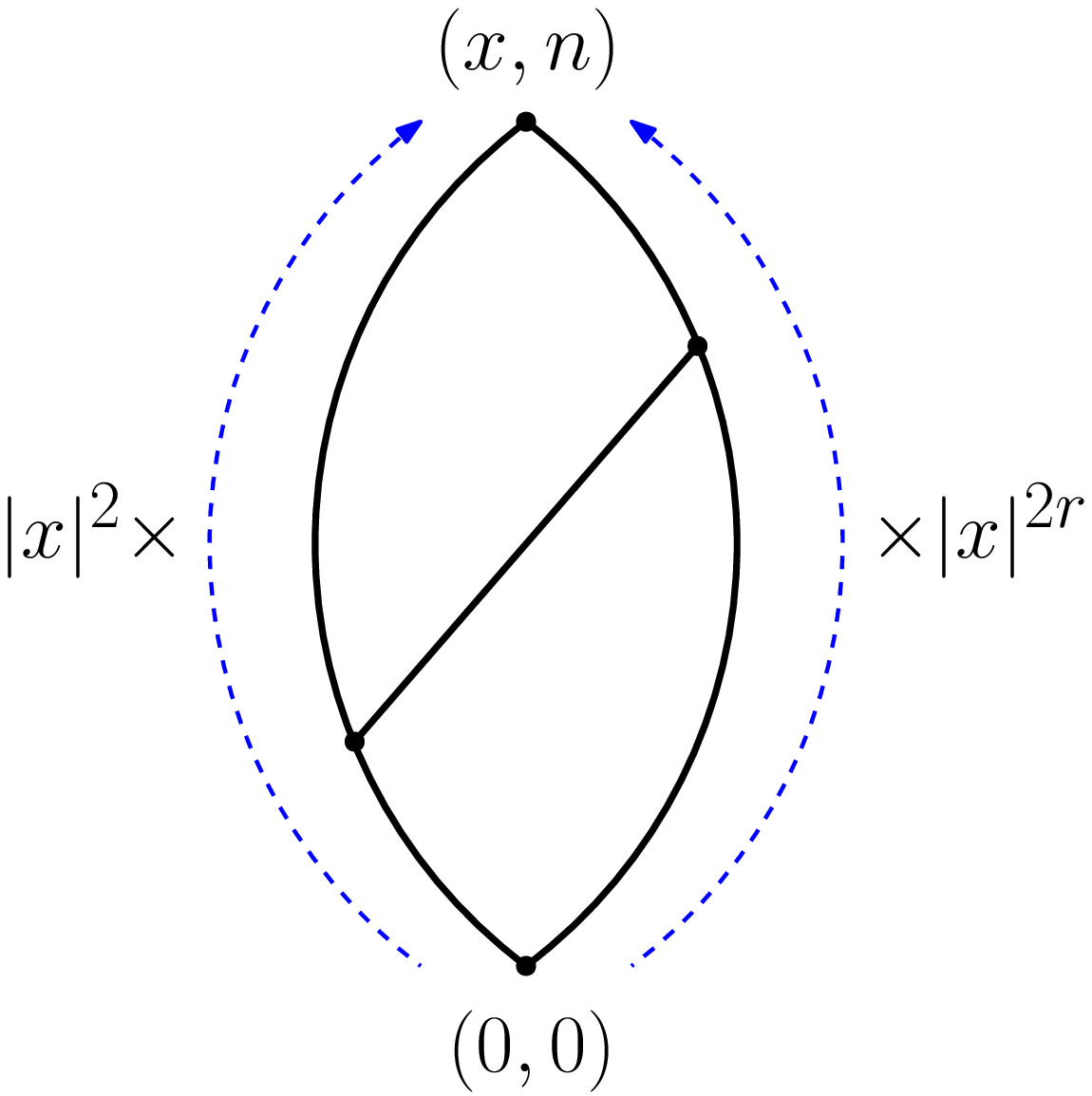}&&
\includegraphics[scale=.40]{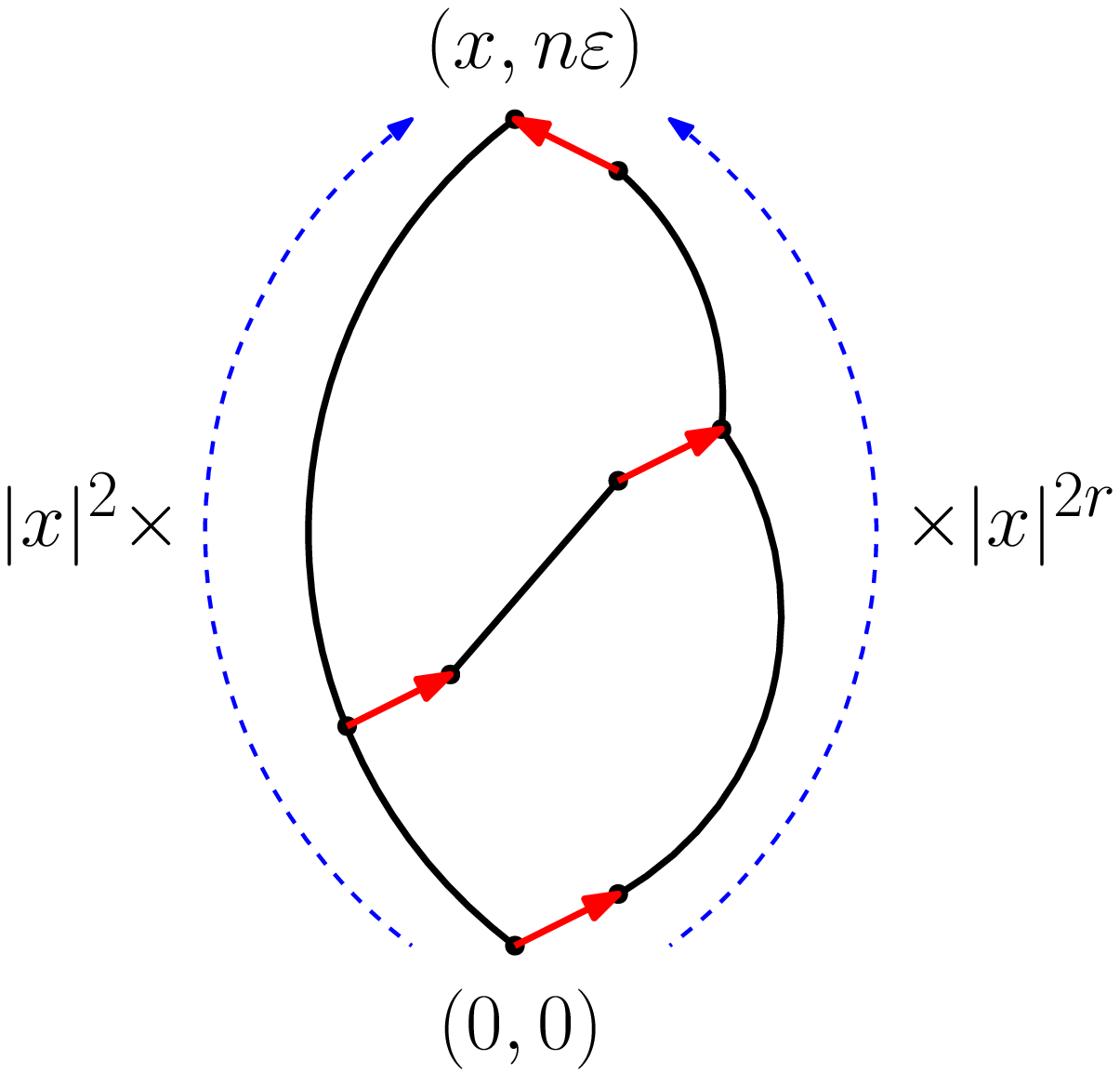}
\end{align*}
\caption{\label{fig:pi1}Allocation of $|x|^2$ and $|x|^{2r}$ to a diagram bounding
$\pi_n^{\scriptscriptstyle(1)}(x)$ for oriented percolation (left)
and the contact process (right).}
\end{figure}

\begin{lemma}\label{lem:pibd}
Assume the same setting as Proposition~\ref{prop:Deltart}, and
let $r\in\N$.
Suppose there is an $\varepsilon$-independent constant $C$ such that
$\sum_{x\in\Zd} |x|^{2r}\tau_n(x) \le C(1+n\varepsilon)^r$ holds for all
$n\ge0$.  Then there is an $\varepsilon$-independent constant $C'$
such that
$\sum_{x\in\Zd} |x|^{2r+2}|\pi_n(x)| \le C'\varepsilon^2(1+n\varepsilon)^{r+1-d/2}$
also holds for all $n\ge0$.
\end{lemma}

\begin{proof}
We make the split $|x|^{2r+2}=|x|^2|x|^{2r}$ and multiply $|x|^2$ on the
left side and $|x|^{2r}$ to the right side of the diagrams bounding
$\pi_n$ (see Figure~\ref{fig:pi1} for an example of a diagram).
Then we use the triangle inequality to decompose
$|x|^2$ along the left side of the diagram and $|x|^{2r}$ along the right side
of the diagram.
The proof of the $r=1$ case of \refeq{pinbd} uses the $\ell_1$ and
$\ell_\infty$ estimates \refeq{taunbd1}--\refeq{taunbd2} for $r=0,1$
to obtain an upper bound of order $\varepsilon^2(1+n\varepsilon)^{1-d/2}$,
where $\varepsilon^2$ arises from the occupied spatial bonds at $(0,0)$ and
$(x,n\varepsilon)$.  By Lemma~\ref{lem:ellnorms} and the hypothesis, we also
have $\sup_x |x|^{2r} \tau_n(x) \le C'(1+n\varepsilon)^{r-d/2}$.
If the same bounds as in the proof of  the $r=1$ case of \refeq{pinbd} are
applied with one line having weight $|x|^{2r}$, then our assumption tells us
that this line contributes an additional factor $(1+m\varepsilon)^r$ where
$m\varepsilon$ is the temporal displacement of the line.  Since $m\le n$,
we obtain an upper bound of order $(1+n\varepsilon)^{r+1-d/2}$, as required.
\end{proof}

\begin{example}
\label{ex:diagram}
We illustrate the diagrammatic estimate underlying the proof of
Lemma~\ref{lem:pibd} with an example for
$\varepsilon=1$ and
$r=2$.
Suppose that $\sum_x |x|^{4}\tau_n(x) \le O(n^2)$, and
consider the left diagram in Figure~\ref{fig:pi1}
weighted with $|x|^6$ and summed over $x$, which is a prototype for
a contribution to  $\sum_x |x|^{6}|\pi_n(x)|$.
We write $|x|^6=|x|^2|x|^4$ and
decompose each factor along the two sides of the diagram using the
triangle inequality.
One contribution that results is
\begin{align}
    X_n &= \sum_{l \le m \le n}\,\sum_{u,v,x}
    |x-u|^4 |x-v|^2 \tau_l(u)\tau_m(v)\tau_{m-l}(v-u)\tau_{n-l}(x-u)\tau_{n-m}(x-v)
    .
\end{align}
By hypothesis and Lemma~\ref{lem:ellnorms},
$\sup_z |z|^{4}\tau_n(z) \le O(n^{2-d/2})$.  Therefore, by \refeq{taunbd1},
\begin{align}
    X_n & \le
    c\sum_{l \le m \le n} \sum_{u}
     \tau_l(u)m^{-d/2}
     \sum_{v}
     \tau_{m-l}(v-u)(n-l)^{2-d/2} \sum_x |x-v|^2\tau_{n-m}(x-v)
     \nnb  &
    \le
    c'\sum_{l \le m \le n}  m^{-d/2}  (n-l)^{2-d/2} (n-m)
    \nnb  &
    \le
    c''n^3
    \sum_{m=1}^{n-1} m^{-d/2} \sum_{l=1}^m (n-l)^{ -d/2}
    .
\end{align}
The elementary
verification that the sum on the right-hand side decays like $n^{-d/2}$ is
carried out in detail in \cite[Example~4.4]{HS03b}.
This gives an overall bound $n^{3-d/2}$ and
illustrates the origin of the bound in the conclusion of
Lemma~\ref{lem:pibd}.  Note that the factor $n^3$ in the overall bound results
from the inequalities $(n-l)^2 \le n^2$ and $n-m \le n$ used above to estimate
the factors arising from the spatial moments.
\end{example}

The next lemma promotes the bounds of Lemma~\ref{lem:pibd} to bounds
on generating functions.

\begin{lemma}\label{lem:Pibd-CP}
Assume the same setting as Proposition~\ref{prop:Deltart}, and
let $r\in\N$.  Suppose there is an $\varepsilon$-independent
constant $C$ such that
$\sum_x |x|^{2r}\tau_n(x) \le C(1+n\varepsilon)^r$ holds for all $n\ge0$.
Then, uniformly in $\varepsilon$,
\begin{align}\lbeq{DeltarPi-1comp}
 |\partial_1^{2r+2} \hat\Pi_z(0)|
  \le\frac{O(\varepsilon^2)}{1-|z|}+
  \1_{r+1>d/2}
  \frac{  O(\varepsilon^{r+3-d/2})}{(1-|z|)^{r+2-d/2}}
  \qquad(|z|<1).
\end{align}
\end{lemma}

\begin{proof}
We drop the argument ``0'' from $\Pihat_z(0)$, and write $a=r+1-d/2$.
As a first step, observe that
\begin{align}
|\partial_1^{2r+2} \hat\Pi_z |
 =\bigg| \sum_{n=2}^\infty\sum_xx_1^{2r+2} \pi_n(x) z^n \bigg|
 \le \sum_{n=2}^\infty\sum_x|x|^{2r+2} |\pi_n(x)| |z|^n.
\end{align}
Therefore, by hypothesis and Lemma~\ref{lem:pibd}
\begin{align}\lbeq{DeltarPi-1}
|\partial_1^{2r+2} \hat\Pi_z | \le O(\varepsilon^2)\sum_{n=2}^\infty
 (1+n\varepsilon)^a|z|^n.
\end{align}
The case $r+1\le d/2$ readily follows from $(1+n\varepsilon)^a\le 1$.

Suppose $r+1>d/2$, so that $a>0$.
Since $(1+x)^a \le 2^a(1+x^a)$
and $|z|\le e^{-(1-|z|)}$, we find that
\begin{align}
|\partial_1^{2r+2} \hat\Pi_z |
& \le O(\varepsilon^2) \sum_{n=2}^\infty   |z|^n
    + O(\varepsilon^2) \sum_{n=2}^\infty (n\varepsilon)^a e^{-n(1-|z|)}
\end{align}
The first term is bounded by $O(\varepsilon^2)(1-|z|)^{-1}$.
For the second term, we write $\theta = 1-|z|$ and use Riemann sum
approximation to see that, as $\theta\downarrow 0$,
\begin{equation}
    \sum_{n=1}^\infty n^a e^{-n\theta}
    = \theta^{-a-1} \sum_{t\in \theta \N} \theta (n\theta )^a e^{-n\theta}
    \sim
    \theta^{-a-1} \Gamma(a+1),
\end{equation}
to complete the proof.
\end{proof}

\begin{lemma}
\label{lem:DeltarPhi}
Suppose that the bound \refeq{DeltarPi-1comp} holds for all $r+1 \le R+1$.
Then, for $2 \le j \le R+1$,
\begin{align}
\lbeq{DeltajPhi}
    |\partial_1^{2j}\hat\Phi_z(0)|
    & \le
    O(\varepsilon) + \frac{O(\varepsilon^2)}{1-|z|}
    +
    \1_{j>d/2}
    \left(
    \frac{O(\varepsilon^{j+2-d/2})}{(1-|z|)^{j+1-d/2}}
    +\sum_{i=\lfloor d/2\rfloor+1}^{j-1}\frac{O(\varepsilon^{i+3-d/2})}{(1-|z|)^{i+1-d/2}}
    \right).
\end{align}
\end{lemma}

\begin{proof}
We again drop the argument ``0'' from the transforms.  We use the
fact that $\partial_1^{2j}\hat q = O(\varepsilon)$ for $j\ge 1$ (but $\hat q=O(1)$).
For $2 \le j \le d/2$,
\begin{align}\lbeq{DeltajPhi1}
|\partial_1^{2j}\hat\Phi_z|&\le  |1+\hat\Pi_z|
 |\partial_1^{2j}\hat q|+ |\partial_1^{2j}\hat\Pi_z|
 \hat q
 +\sum_{i=1}^{j-1}\binom{2j}{2i}|\partial_1^{2i}\hat\Pi_z|
 |\partial_1^{2j-2i}\hat q|
 \nnb
&=O(\varepsilon)+\frac{O(\varepsilon^2)}{1-|z|}.
\end{align}
For $j> d/2$, \refeq{DeltajPhi1} needs
to be modified as
\begin{align}\lbeq{DeltajPhi2}
|\partial_1^{2j}\hat\Phi_z| &\le
O(\varepsilon)+\frac{O(\varepsilon^2)}{1-|z|}
 +\frac{O(\varepsilon^{j+2-d/2})}{(1-|z|)^{j+1-d/2}}
+
\sum_{i=\lfloor d/2\rfloor+1}^{j-1}\binom{2j}{2i}
 |\partial_1^{2i}\hat\Pi_z|
 |\partial_1^{2j-2i}\hat q|
 \nnb
&= O(\varepsilon)+\frac{O(\varepsilon^2)}{1-|z|}
 + \frac{O(\varepsilon^{j+2-d/2})}{(1-|z|)^{j+1-d/2}}
 + \sum_{i=\lfloor d/2\rfloor+1}^{j-1}\frac{O(\varepsilon^{i+3-d/2})}
 {(1-|z|)^{i+1-d/2}} .
\end{align}
This completes the proof.
\end{proof}

\begin{proof}[Proof of Proposition~\ref{prop:Deltart}]
We again drop the argument ``0.''
We use induction on $r\in \N$ to prove \refeq{IRr-1comp},
which asserts that
\begin{align}\lbeq{IRr-bis}
|\partial_1^{2r}\hat t_z|
\le\frac{O(\varepsilon)}{|1-z|^2}\sum_{j=0}^{r-1}
 \frac{ \varepsilon^j }{(1-|z|)^j}
 \qquad (r \ge 1).
\end{align}
By \refeq{IR1},
the initial case $r=1$ is already confirmed.

Suppose that \refeq{IRr-bis} holds for all positive integers $r\le R$.
By Lemma~\ref{lem:DS}, there are $\varepsilon$-independent constants
$C_r$ such that, for all nonnegative integers $n$,
\begin{align}
\sum_x|x|^{2r}\tau_n(x)\le
C_r(1+n\varepsilon)^r.
\end{align}
By Lemma~\ref{lem:Pibd-CP},
this provides the estimate \refeq{DeltarPi-1comp} on
$\partial_1^{2r+2}\Pihat_z$ for all $r \le R$.

Suppose first that $R+1\le d/2$.  By \refeq{Deltar} with $r=R+1$,
\begin{align}
\partial_1^{2R+2}\hat t_z
&=\frac{1}{1-\hat\Phi_z}\bigg(\partial_1^{2R+2}\hat\Pi_z
 +\sum_{j=1}^{R+1} \binom{2R+2}{2j}\partial_1^{2j}\hat\Phi_z\,
 \partial_1^{2R+2-2j}\hat t_z \bigg).
\end{align}
By Corollary~\ref{cor:Pibds} and Lemmas~\ref{lem:Pibd-CP}--\ref{lem:DeltarPhi}, followed  by application
of the induction hypothesis,
\begin{align}\lbeq{DeltaR+1t1}
 |\partial_1^{2R+2}\hat t_z|
    & \le
    \frac{O(1)}{|1-z|}
    \bigg(
     |\partial_1^{2R+2}\hat\Pi_z|
    +
    |\partial_1^2\hat\Phi_z|\,|\partial_1^{2R}\hat t_z|
    +\sum_{j=2}^R |\partial_1^{2j} \hat\Phi_z| \, |\partial_1^{2R+2-2j}\hat t_z|
    + |\partial_1^{2R+2}\hat\Phi_z| \, |\hat t_z|
    \bigg)
    \nnb
    &\le
    \frac{O(1)}{|1-z|}\Bigg(\frac{\varepsilon^2}{1-|z|}
    + \frac{\varepsilon^2}{|1-z|^2} \sum_{j=0}^{R-1}
     \frac{ \varepsilon^j }{(1-|z|)^j}\nnb
&\qquad+\bigg(\varepsilon+\frac{\varepsilon^2}{1-|z|}\bigg)
    \frac{\varepsilon}{|1-z|^2}\sum_{j=2}^R
    \sum_{i=0}^{R-j}
    \frac{ \varepsilon^i }{(1-|z|)^i}
 + \bigg(\varepsilon+\frac{\varepsilon^2}{1-|z|}\bigg)
    \frac{1}{|1-z|} \Bigg)
    .
\end{align}
It is then an exercise in bookkeeping to verify that this implies that,
as required,
\begin{align}
\lbeq{correctform}
|\partial_1^{2R+2}\hat t_z|
&\le
\frac{O(\varepsilon)}{|1-z|^2}\sum_{j=0}^R\frac{\varepsilon^j}
 {(1-|z|)^j}
 .
\end{align}

Suppose finally that $R+1>d/2$.
Then \refeq{DeltaR+1t1} is modified due to the extra terms in
\refeq{DeltarPi-1comp} and \refeq{DeltajPhi}.
The contribution due to the extra term in $\partial_1^{2R+2}\hat\Pi_z$
can be estimated by
\begin{equation}
    \frac{O(1)}{|1-z|}\frac{\varepsilon^{R+3-d/2}}{(1-|z|)^{R+2-d/2}}
    =\frac{O(\varepsilon)}{|1-z|}
    \frac{\varepsilon}{1-|z|}
    \frac{\varepsilon^{R+1-d/2}}{(1-|z|)^{R+1-d/2}}.
\end{equation}
The above is of the correct form to advance the induction if $d>4$ is even.
If instead $d$ is odd (in which case $R+1\ge(d+1)/2$), then we use
$1 \le x^{-1/2}+x^{1/2}$ for $x=\varepsilon(1-|z|)^{-1} >0$ to see that
\begin{align}\lbeq{odd}
\frac{\varepsilon^{R+1-d/2}}{(1-|z|)^{R+1-d/2}}
 \le\frac{\varepsilon^{R+1-(d+1)/2}}{(1-|z|)^{R+1-(d+1)/2}}
 +\frac{\varepsilon^{R+1-(d-1)/2}}{(1-|z|)^{R+1-(d-1)/2}},
\end{align}
and we obtain a result of the correct form to advance the induction also in this case.
We write the absolute value of the
additional term  in \refeq{DeltajPhi} as $X_j$.  It contributes
to the last two terms of the first line of \refeq{DeltaR+1t1} an amount
\begin{equation}
\lbeq{finally}
    \frac{O(1)}{|1-z|} \left(
    \sum_{j=\lfloor d/2\rfloor+1}^R X_j |\partial_1^{2R+2-2j}\hat t_z| + X_{R+1}|\hat t_z|
    \right).
\end{equation}
The $X_{R+1}$ term is bounded by
\begin{align}
\label{e:XR}
        \frac{O(\varepsilon)}{|1-z|^2}
        \left( \frac{\varepsilon^{R+2-d/2}}{(1-|z|)^{R+2-d/2}}
    +\sum_{i=\lfloor d/2\rfloor+1}^{R}\frac{\varepsilon^{i+2-d/2}}{(1-|z|)^{i+1-d/2}}
    \right).
\end{align}
If $d>4$ is even, then this is bounded by the right-hand side
of \refeq{correctform}.  If $d$ is odd, then we again apply \refeq{odd} to
obtain an estimate that is appropriate to advance the induction.
Finally, for the sum in \refeq{finally},
we use the induction hypothesis to obtain an upper bound
\begin{align}
    &\frac{O(\varepsilon)}{|1-z|^{2}} \frac{\varepsilon}{|1-z|}
    \sum_{j=\lfloor d/2\rfloor + 1}^R
    \left(
    \frac{\varepsilon^{j+ 1-d/2}}{(1-|z|)^{j+1-d/2}}
    +
    \sum_{i=\lfloor d/2\rfloor+1}^{j-1}
    \frac{\varepsilon^{i+ 2-d/2}}{(1-|z|)^{i+1-d/2}}
    \right)
    \sum_{ \ell=0}^{R-j}\frac{\varepsilon^{ \ell}}{(1-|z|)^{ \ell}} .
\end{align}
This again has the correct form to advance the induction, again with the
distinction between even and odd $d>4$.
For example, the $j=R$ term in the first term in the first sum, and the $\ell=R-j$
term in the last sum, combine to give a contribution
\begin{equation}
    \frac{O(\varepsilon)}{|1-z|^{2}} \frac{\varepsilon}{|1-z|}
    \frac{\varepsilon^{R+ 1-d/2}}{(1-|z|)^{R+1-d/2}},
\end{equation}
which is bounded above by the first term in \refeq{XR} and we have seen that this bound is
acceptable.
This completes the proof.
\end{proof}

\section{Lattice trees: proof of Theorem~\ref{thm:lt}}
\label{sec:lt}

We sketch the proof, and only point out where it differs
from the proof of Theorem~\ref{thm:op}.
We assume henceforth that $d>8$ and $L$ is
sufficiently large.

Recall the definitions of Section~\ref{sec:lt-intro}, and, as in \refeq{tzdef},
let
\begin{equation}
  t_z(x) = \sum_{n=0}^\infty \tau_n(x) z^n
  \qquad(|z|<1) .
\end{equation}
The lace expansion gives (see, e.g., \cite[(4.7)--(4.9)]{DS98})
\begin{equation}
    \hat t_z (k)  = \frac{\hat h_z(k)}{1-\hat \Phi_z(k)},
\end{equation}
with
\begin{equation}
  \hat h_z(k) = g_{p_c} + \hat\Pi_z(k),
  \qquad
  \hat \Phi_z(k) =  z p_c\hat{D}(k)\hat h_z(k).
\end{equation}
The critical 1-point function $g_{z_c}$ is a constant
in the interval $[1,4]$ \cite{HS90b},
and the coefficients of the power series $\Pi_z(x)$ are given by
\begin{equation}
  \Pi_z(x) = \sum_{n=0}^\infty\pi_n(x) z^n.
\end{equation}

The following proposition plays the role of
Proposition~\ref{prop:HSbds}; note that the power
$n^{r-d/2}$ of \refeq{pinbd} is replaced by $n^{r-(d-4)/2}$ for lattice trees.
This reflects the increase in the upper critical dimension
from $4$ for oriented percolation to $8$ for lattice trees.

\begin{prop}
\label{prop:Holm08bds}
Let $d>8$ and fix any $\delta>0$.
There is an $L_0>0$ and
finite $C$ (independent of $L$) and $C_L$ (dependent on $L$)
such that for $L \ge L_0$ and for $n \ge 1$,
\begin{equation}
 \lbeq{taunbd-lt}
    \sum_{x\in\Zd} |x|^{2r} \tau_n(x)  \leq C_L
    n^{r},    \qquad
    \sup_{x\in\Zd} |x|^{2r} \tau_n(x) \leq C_L
    n^{r-d/2}
    \qquad
    (r=0,1),
\end{equation}
\begin{equation}
    \lbeq{pinbd-lt}
        \sum_{x\in\Zd} |x|^{2r}|\pi_n(x) | \leq  CL^{2r-d+\delta} n^{r-(d-4)/2}
    \quad (r=0,1,2).
\end{equation}
\end{prop}

\begin{proof}
The bound \refeq{pinbd-lt} is given in \cite[Proposition~5.1]{Holm08}.
The bounds on the sum in \refeq{taunbd-lt} follow from
\cite[Theorem~3.7]{Holm08},
and then the bounds on the supremum follow from
Lemma~\ref{lem:ellnorms} (which applies equally well to lattice trees).
\end{proof}

By \refeq{pinbd-lt}, $|\hat\Pi_z(k)|$
and $|\Delta \hat\Pi_z(k)|$ are uniformly bounded in $|z|\le 1$.
The bound on the 1-point function mentioned above then gives a bound on $\hat h_z(k)$
in the same closed disk.  By the bound on $1-\hat{\Phi}_z(0)$ provided by
\cite[Lemma~5.3(ii)]{DS98} (their $\zeta$ is our $z$), as
in \refeq{IR0} we obtain
\begin{equation}
    |\hat t_z (0) |
    \le \frac{C}{|1-z|}
    \qquad(|z|<1);
\end{equation}
in fact much more is known (see \cite[(2.5)]{DS98}).
We have already noted that Lemma~\ref{lem:ellnorms} applies to lattice trees.
Lemma~\ref{lem:pibd} is replaced by the following lemma, whose proof
we discuss below.

\begin{lemma}
\label{lem:pibd-lt}
Assume the same setting as Proposition~\ref{prop:Holm08bds},
and let $r\in\N$.  Suppose there is a finite $C$ such that
$\sum_{x\in\Zd} |x|^{2r}\tau_n(x) \le C n^{r}$ holds for all $n\ge1$.
Then there is a finite $C'$ such that
$\sum_{x\in\Zd} |x|^{2r+2}|\pi_n(x)| \le C' n^{r+1-(d-4)/2}$ holds for all
$n\ge1$.
\end{lemma}

The conclusion of
Lemma~\ref{lem:Pibd-CP} is then replaced by
\begin{equation}
    \sum_{n=1}^\infty \sum_{x\in\Zd} |x|^{2r+2}|\pi_n(x)|  \, |z|^n
    \le O\big( (1-|z|)^{-(r-(d-8)/2) \vee 1} \big) \qquad (|z|<1).
\end{equation}
The proof of Proposition~\ref{prop:Deltart} is essentially
the same but is simpler because now we can
set $\varepsilon=1$, and it applies also to prove Theorem~\ref{thm:lt}.
It remains only to discuss the proof of Lemma~\ref{lem:pibd-lt}.
This is carried out in detail for $r=1$ in the proof of
\cite[Proposition~5.1]{Holm08}, and  for $r=2$
in the proof of \cite[(9.32)]{HP19}.
The same method applies more generally to handle higher values of $r$.
Briefly, the proof goes as follows.

\begin{figure}
\begin{align*}
\includegraphics[scale=.4]{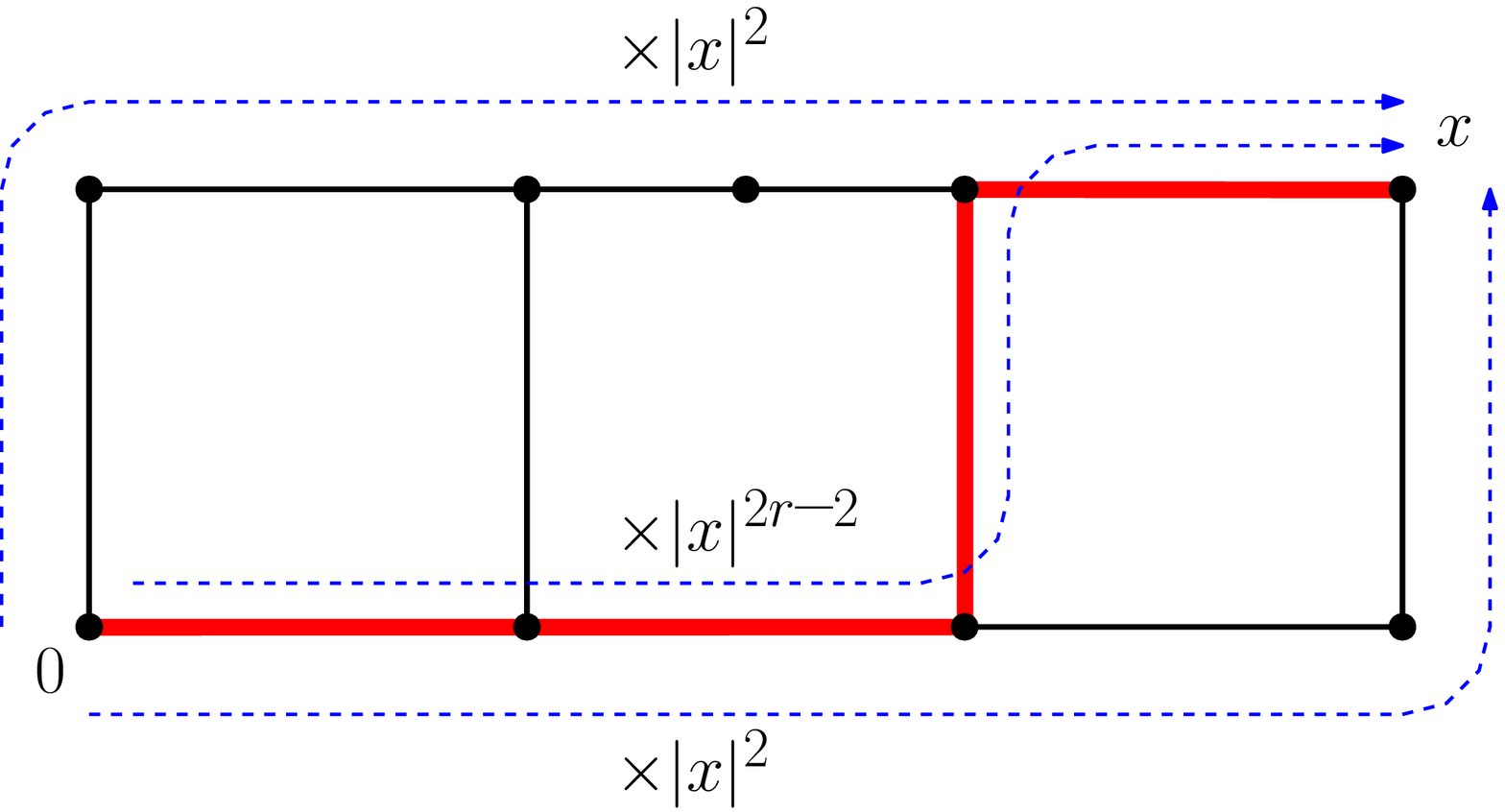}
\end{align*}
\caption{\label{fig:LTpi5} Allocation of $|x|^2$ and $|x|^{2r-2}$
to a diagram bounding $\pi_n^{\scriptscriptstyle(3)}(x)$ for lattice trees.
The backbone is represented by the sequence of bold lines (in red).}
\end{figure}

\begin{proof}[Proof of Lemma~\ref{lem:pibd-lt}]
The diagrammatic estimate in the proof of
Lemma~\ref{lem:pibd} must be replaced by an estimate for the diagrams
that arise for lattice trees (the diagrams are discussed in
\cite[Section~9.2]{HP19} ---
see, in particular, \cite[Figure~2]{HP19}).
We divide $|x|^{2r+2}$ as $|x|^2|x|^{2r-2}|x|^2$, distribute
one $|x|^2$ factor along the top of a diagram via the
triangle inequality, and distribute the other  $|x|^2$ factor
along the bottom of the diagram (see Figure~\ref{fig:LTpi5} for an example
of a 3-loop diagram).  This leads to terms with
one line on the top of the diagram weighted with the displacement
squared, and one line on the bottom similarly weighted.
The factor $|x|^{2r-2}$ is distributed along the backbone, which
includes lines on top and bottom of the diagram, which may or may
not be already weighted with the displacement squared.
Thus, altogether, we have one line weighted with $|y|^2$
and a different line (which must lie on the backbone) weighted with $|y|^2|y|^{2r-2}=|y|^{2r}$,
or we have two lines weighted with $|y|^2$
and a third line (which must lie on the backbone) weighted with $|y|^{2r-2}$.
The case $r=1$ is handled in
\cite[Proposition~5.1]{Holm08} by using the bounds \refeq{taunbd-lt}
on the backbone lines.  For $r>1$, we first apply H\"older's inequality
(as in \refeq{Holder1}) to see
that the hypothesis on the $(2r)^\text{th}$ moment of $\tau_n$
implies $\sum_{x\in\Zd}|x|^{2r-2}\tau_n(x)=O(n^{r-1})$.
With Lemma~\ref{lem:ellnorms}, these bounds on the $(2r)^\text{th}$ and
$(2r-2)^\text{th}$ moments of $\tau_n$ imply corresponding $\ell_\infty$ bounds.
Together, these imply
that the estimate
for the $(2r+2)^\text{th}$ moment of $\pi_n$ will be at most $n^{r-1}$ times larger
than the fourth moment estimate of \refeq{pinbd-lt}, i.e.,
\begin{equation}
  \sum_{x\in\Zd} |x|^{2r+2}|\pi_n(x) | \leq   n^{r-1} O(n^{2-(d-4)/2})
  = O( n^{r+1-(d-4)/2}),
\end{equation}
as required.
\end{proof}

\section*{Acknowledgements}
We thank Mark Holmes and Ed Perkins for bringing this problem to our
attention and for several helpful conversations.
The work of AS was supported by JSPS KAKENHI Grant Number 18K03406.
The work of GS was supported in part by NSERC of Canada
and by a grant from the Simons Foundation.
GS would like to thank the Isaac Newton Institute for Mathematical
Sciences for support and hospitality during the programme ``Scaling limits,
rough paths, quantum field theory'' when work on this paper was
undertaken; this work was supported by EPSRC Grant Numbers EP/K032208/1
and EP/R014604/1.


\begin{thebibliography}{10}

\bibitem{BW98}
D.J. Barsky and C.C. Wu.
\newblock Critical exponents for the contact process under the triangle
  condition.
\newblock {\em J. Stat. Phys.}, {\bf 91}:95--124, (1998).

\bibitem{BG91}
C.~Bezuidenhout and G.~Grimmett.
\newblock Exponential decay for subcritical contact and percolation processes.
\newblock {\em Ann. Probab.}, {\bf 19}:984--1009, (1991).

\bibitem{CS08}
L.-C. Chen and A.~Sakai.
\newblock Critical behavior and the limit distribution for long-range oriented
  percolation. {I}.
\newblock {\em Probab. Theory Related Fields}, {\bf 142}:151--188, (2008).

\bibitem{CS09a}
L.-C. Chen and A.~Sakai.
\newblock Critical behavior and the limit distribution for long-range oriented
  percolation. {II}: {Spatial} correlation.
\newblock {\em Probab. Theory Related Fields}, {\bf 145}:435--458, (2009).

\bibitem{CS11}
L.-C. Chen and A.~Sakai.
\newblock Asymptotic behavior of the gyration radius for long-range
  self-avoiding walk and long-range oriented percolation.
\newblock {\em Ann.\ Probab.}, {\bf 39}:507--548, (2011).

\bibitem{CDP00}
J.T. Cox, R.~Durrett, and E.A. Perkins.
\newblock Rescaled voter models converge to super-{Brownian} motion.
\newblock {\em Ann. Probab.}, {\bf 28}:185--234, (2000).

\bibitem{DS98}
E.~Derbez and G.~Slade.
\newblock The scaling limit of lattice trees in high dimensions.
\newblock {\em Commun.\ Math.\ Phys.}, {\bf 193}:69--104, (1998).

\bibitem{DP99}
R.~Durrett and E.A. Perkins.
\newblock Rescaled contact processes converge to super-{Brownian} motion in two
  or more dimensions.
\newblock {\em Probab. Theory Related Fields}, {\bf 114}:309--399, (1999).

\bibitem{FO90}
P.~Flajolet and A.~Odlyzko.
\newblock Singularity analysis of generating functions.
\newblock {\em SIAM J. Disc. Math.}, {\bf 3}:216--240, (1990).

\bibitem{HS90a}
T.~Hara and G.~Slade.
\newblock Mean-field critical behaviour for percolation in high dimensions.
\newblock {\em Commun. Math. Phys.}, {\bf 128}:333--391, (1990).

\bibitem{HS90b}
T.~Hara and G.~Slade.
\newblock On the upper critical dimension of lattice trees and lattice animals.
\newblock {\em J. Stat. Phys.}, {\bf 59}:1469--1510, (1990).

\bibitem{HS92c}
T.~Hara and G.~Slade.
\newblock The number and size of branched polymers in high dimensions.
\newblock {\em J. Stat. Phys.}, {\bf 67}:1009--1038, (1992).

\bibitem{HH17book}
M.~Heydenreich and R.~van~der Hofstad.
\newblock {\em Progress in High-Dimensional Percolation and Random Graphs}.
\newblock Springer International Publishing Switzerland, (2017).

\bibitem{HHS08}
R. van der Hofstad, M. Holmes, and G. Slade.
\newblock An extension of the inductive approach to the lace expansion.
\newblock {\em Elect. Comm. Probab.}, {\bf 13}:291--301, (2008).

\bibitem{HS04z}
R.~van~der Hofstad and A.~Sakai.
\newblock Gaussian scaling for the critical spread-out contact process above
  the upper critical dimension.
\newblock {\em Electr. J. Probab.}, {\bf 9}:710--769, (2004).

\bibitem{HS05y}
R.~van~der Hofstad and A.~Sakai.
\newblock Critical points for spread-out self-avoiding walk, percolation
		and the contact process.
\newblock {\em Probab. Theory Related Fields}, {\bf 132}:438--470, (2005).

\bibitem{HS10}
R.~van~der Hofstad and A.~Sakai.
\newblock Convergence of the critical finite-range contact process to
  super-{Brownian} motion above the upper critical dimension.
\newblock {\em Electr. J. Probab.}, {\bf 15}:801--894, (2010).

\bibitem{HS02}
R. van der Hofstad and G. Slade.
\newblock A generalised inductive approach to the lace expansion.
\newblock {\em Probab. Th. Rel. Fields}, {\bf 122}:389--430, (2002).

\bibitem{HS03b}
R.~van~der Hofstad and G.~Slade.
\newblock Convergence of critical oriented percolation to super-{Brownian}
  motion above $4+1$ dimensions.
\newblock {\em Ann.\ Inst.\ H.\ Poincar\'e Probab.\ Statist.}, {\bf
  39}:413--485, (2003).

\bibitem{Holm08}
M.~Holmes.
\newblock Convergence of lattice trees to super-{B}rownian motion above the
  critical dimension.
\newblock {\em Electr.\ J.\ Probab.}, {\bf 13}:671--755, (2008).

\bibitem{Holm16}
M.~Holmes.
\newblock Backbone scaling for critical lattice trees in high dimensions.
\newblock {\em J. Phys. A: Math. Theor.}, {\bf 49}:314001, (2016).

\bibitem{HP19}
M.~Holmes and E.~Perkins.
\newblock On the range of lattice models in high dimensions - extended version.
\newblock Preprint, \url{https://arxiv.org/abs/1806.08497}, (2018).

\bibitem{HP20}
M.~Holmes and E.~Perkins.
\newblock On the range of lattice models in high dimensions.
\newblock Preprint, (2019).

\bibitem{NY93}
B.G. Nguyen and W-S. Yang.
\newblock Triangle condition for oriented percolation in high dimensions.
\newblock {\em Ann.\ Probab.}, {\bf 21}:1809--1844, (1993).

\bibitem{NY95}
B.G. Nguyen and W-S. Yang.
\newblock Gaussian limit for critical oriented percolation in high dimensions.
\newblock {\em J. Stat. Phys.}, {\bf 78}:841--876, (1995).

\bibitem{Saka01}
A.~Sakai.
\newblock Mean-field critical behavior for the contact process.
\newblock {\em J.\ Stat.\ Phys.}, {\bf 104}:111--143, (2001).

\bibitem{Slad02}
G.~Slade.
\newblock Scaling limits and super-{Brownian} motion.
\newblock {\em Not. Am. Math. Soc.}, {\bf 49}(9):1056--1067, (2002).

\bibitem{Slad06}
G.~Slade.
\newblock {\em The Lace Expansion and its Applications.}
\newblock Springer, Berlin, (2006).
\newblock Lecture Notes in Mathematics Vol. 1879. Ecole d'Et\'{e} de
  Probabilit\'{e}s de Saint--Flour XXXIV--2004.

\bibitem{Wu95}
C.C.~Wu.
\newblock The contact process on a tree: {Behavior} near the first phase
transition.
\newblock {\em Stoch.\ Proc.\ Appl.}, {\bf 57}:99--112, (1995).


\end{thebibliography}
\end{document}